\documentclass{amsart}
\usepackage{latexsym,amssymb,amsthm,amsmath,verbatim}
\usepackage{lscape,pdflscape}
\usepackage{enumerate}
\usepackage{tikz}
\usetikzlibrary{arrows,positioning,chains,matrix,scopes}

\begin{document}

\newtheorem{theorem}{Theorem}[section]
\newtheorem{proposition}[theorem]{Proposition}
\newtheorem{definition}[theorem]{Definition}
\newtheorem{corollary}[theorem]{Corollary}
\newtheorem{lemma}[theorem]{Lemma}
\newtheorem{question}[theorem]{Question}

\theoremstyle{definition}
\newtheorem{remark}{Remark}
\newtheorem{example}{Example}

\newcommand{\pr}[1]{\left\langle #1 \right\rangle}
\newcommand{\mR}{\mathcal{R}}
\newcommand{\RR}{\mathbb{R}}
\newcommand{\QQ}{\mathbb{Q}}
\newcommand{\mA}{\mathcal{A}}
\newcommand{\mE}{\mathcal{E}}
\newcommand{\mV}{\mathcal{V}}
\newcommand{\mF}{\mathcal{F}}
\newcommand{\mU}{\mathcal{U}}
\newcommand{\mP}{\mathcal{P}}
\newcommand{\mT}{\mathcal{T}}
\newcommand{\mB}{\mathcal{B}}
\newcommand{\mK}{\mathcal{K}}
\newcommand{\C}{\mathrm{C}}
\newcommand{\mC}{\mathcal{C}}
\newcommand{\mO}{\mathcal{O}}
\newcommand{\mM}{\mathcal{M}}

\newcommand{\D}{\mathrm{D}}
\newcommand{\0}{\mathrm{o}}
\newcommand{\OD}{\mathrm{OD}}
\newcommand{\Do}{\D_\mathrm{o}}
\newcommand{\sone}{\mathsf{S}_1}
\newcommand{\gone}{\mathsf{G}_1}
\newcommand{\sfin}{\mathsf{S}_\mathrm{fin}}
\newcommand{\Em}{\longrightarrow}
\newcommand{\menos}{{\setminus}}
\newcommand{\w}{{\omega}}

\title{Productively countably tight spaces of the form $C_k(X)$}
\author[L. F. Aurichi]{Leandro F. Aurichi$^1$}
\thanks{$^1$ Supported by FAPESP (2013/05469-7)}
\address{Instituto de Ci\^encias Matem\'aticas e de Computa\c c\~ao,
Universidade de S\~ao Paulo, Caixa Postal 668,
S\~ao Carlos, SP, 13560-970, Brazil}
\email{aurichi@icmc.usp.br}

\author[R. M. Mezabarba]{Renan M. Mezabarba$^2$}
\thanks{$^2$ Supported by CAPES (DS-7346753/M)}
\address{Instituto de Ci\^encias Matem\'aticas e de Computa\c c\~ao,
Universidade de S\~ao Paulo, Caixa Postal 668,
S\~ao Carlos, SP, 13560-970, Brazil}
\email{rmmeza@icmc.usp.br}

\keywords{topological games, selection principles, productively countably tightness, Alster spaces, $G_\delta$-topology, bornology}

\subjclass[2010]{Primary 54D20; Secondary 54G99, 54A10}

\maketitle

\begin{abstract}
Some results in $C_k$-theory are obtained with the use of bornologies. We investigate under which conditions the space of the continuous real functions with the compact-open topology is a productively countably tight space, which yields some applications on Alster spaces.

\end{abstract}

\section{Introduction}

Along this work, $\kappa$ and $\lambda$ denote infinite cardinals. By $C_p(X)$ we mean the space of the continuous real functions on $X$ with the topology of the pointwise convergence, while $C_k(X)$ denotes the space of the continuous real functions on $X$ with the compact-open topology.

Recall that for a topological space $X$ and a point $x\in X$, the {\bf tightness of $X$ at $x$}, denoted by $t(x,X)$, is the least cardinal $\kappa$ with the property that if $x\in \overline{A}$ for any $A\subset X$, then there exists a $B\in[A]^{\leq\kappa}$ such that $x\in\overline{B}$. The {\bf tightness of the space $X$} is the supremum of all cardinals $t(x,X)$ for $x\in X$, that we denote by $t(X)$.

Following Arhangel'skii in \cite{Arhan1}, we say that a topological space $X$ is {\bf productively $\kappa$-tight at $x\in X$} if, for any space $Y$ with $t(Y)\leq \kappa$, one has $t((x,y),X\times Y)\leq \kappa$ for any $y\in Y$, and we denote this by $\kappa\in Sp(x,X)$. Naturally, the space $X$ is {\bf productively $\kappa$-tight} if $\kappa\in Sp(x,X)$ for all $x$ in $X$, and we write $\kappa\in Sp(X)$ to denote that.

In \cite{Arhan1}, Arhangel'skii gave an internal characterization to the productively $\kappa$-tightness property of a Tychonoff space $X$ by using the concept of $\kappa$-singular families. A family $\mP$ of collections of subsets of $X$ is {\bf $\kappa$-singular} at $x\in X$ if the following holds:
\begin{enumerate}[(a)]
\item for each $\xi\in \mP$ there exists $A\in\xi$ such that $|A|\leq \kappa$;
\item each $\xi\in\mP$ is centered (has the finite intersection property);
\item for any $O\subset X$ open with $x\in O$ there is some $\xi \in \mP$ and $A\in \xi$ such that $A\subset O$;
\item for any $\mE\subset \mP$ with $|\mE|\leq \kappa$ it is possible to choose $A(\xi)\in \xi$ for each $\xi\in \mE$ such that $x\not\in \overline{\bigcup_{\xi\in\mE}A(\xi)}$.
\end{enumerate}

Note that in the presence of conditions $(b),(c)$ and $(d)$, $(a)$ is equivalent to ask condition $(a)'$: every member of any $\xi\in\mP$ is a subset of $X$ with cardinality at most $\kappa$. Then, we say that $X$ is $\kappa$-singular at $x\in X$ if there exists a family $\mP$ $\kappa$-singular at $x$. With this terminology we may state Arhangel'skii characterization for $Sp(x,X)$:

\begin{theorem}[Arhangel'skii, \cite{Arhan1}, Theorem 3.5]\label{ArhanSp}

For a Tychonoff space $X$ and $x\in X$, $\kappa\in Sp(x,X)$ if, and only if, $X$ is not $\kappa$-singular at $x$.

\end{theorem}

Actually, Arhangel'skii has obtained a third equivalence in the above result, related with the tightness of $x$ in any Hausdorff compactification of $X$. Using this other characterization, Uspenskii showed the following:

\begin{theorem}[Uspenskii, \cite{Uspenskii}, Theorem 1]\label{Uspenskii} For a Tychonoff space $X$, $\kappa\in Sp(C_p(X))$ if, and only if, every open covering $\mU$ for $X_\kappa$ has a subcovering $\mU'$ with cardinality less than or equal to $\kappa$, where $X_{\kappa}$ is the topological space obtained by declaring open the $G_{\kappa}$-sets of $X$.
\end{theorem}

The main result of this work is a generalization of the last result, by using the concepts of bornologies and avoiding the compactification argument of Uspenskii.

\section{Remarks about bornologies}

A {\bf bornology} $\mB$ on a topological space $X$ is an ideal of subsets of $X$ that covers the space. By a (compact) base $\mB'$ for a bornology $\mB$ on $X$, we mean a subset of $\mB$ that is cofinal with respect to inclusion (such that all its elements are compact).

For a topological space $X$ and a bornology $\mB$ on $X$ with compact base, we call the {\bf topology of uniform convergence on} $\mB$, denoted by $\mathcal{T}_{\mB}$, as the topology on $C(X)$ having as a neighborhood base at each $f\in C(X)$ the sets of the form $$\langle B,\varepsilon\rangle[f]:=\{g\in C(X):\forall x\in B(|f(x)-g(x)|<\varepsilon)\},$$
for $B\in\mB$ and $\varepsilon>0$. By $C_{\mB}(X)$ we mean the space $(C(X),\mathcal{T}_{\mB})$.

\begin{remark}
In fact, since $\mT_{\mB}$ is obtained from a separating uniformity over $C(X)$, it follows that $C_{\mB}(X)$ is a Tychonoff space (see McCoy and Ntantu \cite{McCoy}).
\end{remark}

We say that $\mC$ is a $\mB$-covering for $X$ if for every $B\in\mB$ there is a $C\in\mC$ such that $B\subset C$. Following the notation of Caserta \emph{et al.} (\cite{Caserta2}), we denote by $\mathcal{O}_{\mB}$ the collection of all open $\mB$-coverings for $X$. When $\mU\in\mO_{\mB}$ is such that $X\not\in\mU$, we say that $\mathcal{U}$ is nontrivial; note that if $\mU\in\mO_{\mB}$ is nontrivial, then $\mathcal{U}\setminus F$ is an open $\mB$-covering for $X$ for any $F\in[\mathcal{U}]^{<\omega}$.

Also, we denote by $l_{\mB}(X)$ the {\bf $\mB$-Lindel\"of degree of $X$}, that is the smallest transfinite cardinal $\kappa$ such that for every open $\mB$-covering for $X$ there exists a $\mB$-subcovering $\mU'\subset \mU$ with $|\mU'|\leq \kappa$.

\begin{example}
The main examples of bornologies with compact base on a topological space $X$ are the bornologies $\mathcal{F}=[X]^{<\omega}$ and $\frak{K}=\{A\subset X:\exists K\subset X$ compact and $A\subset K\}$ $-$ if $X$ is a Hausdorff space, then $\frak{K}=\{A\subset X:\overline{A}$ is compact$\}$. For $\mB=\mathcal{F}$, one has $C_{\mF}(X)=C_p(X)$ and the $\mF$-coverings are usually called by $\omega$-coverings of $X$; we denote by $\Omega$ the collection of all open $\omega$-coverings. Also, if $X$ is Hausdorff, then for $\mB=\frak{K}$ it follows that $C_{\frak{K}}(X)=C_{k}(X)$ and the $\frak{K}$-coverings for $X$ are the so called $K$-coverings for $X$; the family of all $K$-coverings will be denoted by $\mK$.
\end{example}

\section{Some results with bornologies}

Let $\alpha\geq \omega$ be an ordinal. Recall that the game $\gone^{\alpha}(\mA,\mC)$ denotes the two players game played as follows: for every inning $\gamma<\alpha$, player I chooses an element $A_\gamma\in\mA$, and then player II picks an $a_\gamma\in A_\gamma$; player II wins if $\{a_\gamma:\gamma<\alpha\}\in \mC$; we denote by $\gone(\mA,\mC)$ when $\alpha=\omega$. In the following, we should use the families $\mO_{\mB}$ as well as the family $\Omega_x:=\{A\subset X\setminus\{x\}:x\in \overline{A}\}$. We denote by $\0$ the zero function and the open interval $\left(-\frac{1}{n+1},\frac{1}{n+1}\right)$ by $I_n$, for all $n\in\omega$.

The next Lemma enable us to translate some closure properties of $C_{\mB}(X)$ as ``$\mB$-covering'' properties of $X$, and \emph{vice-versa}.

\begin{lemma}
Let $X$ be a Tychonoff space and let $\mB$ be a bornology with compact base on $X$.
\begin{enumerate}[(a)]
\item If $\mU\in\mO_{\mB}$ is nontrivial, then $A=\{f\in C_{\mB}(X):\exists U\in\mU(f\upharpoonright X\setminus U\equiv 1)\}\in\Omega_{\0}$.
\item Let $A\subset C_{\mB}(X)$, $n\in\omega$ and let $\mU=\{f^{-1}[I_n]:f\in A\}$. If $\0\in\overline{A}$, then $\mU\in\mO_{\mB}$ $($and possibly $X\in\mU$$)$.
\end{enumerate}
\end{lemma}

\begin{proof}
Essentially the same proof of Lemma 2.2 of Caserta \emph{et al.} \cite{Caserta2}.
\end{proof}

The following Theorem is based on a result from Scheepers (\cite{Scheepers1}, Theorem 23):

\begin{theorem}\label{Scheepers1}
Let $X$ be a Tychonoff space and let $\mB$ be a bornology with compact base on $X$. Player II  has a winning strategy
in $\gone(\mathcal{O}_{\mB},\mathcal{O}_{\mB})$ played on $X$  if, and only if,  player
II has a winning strategy in
$\gone(\Omega_{\0},\Omega_{\0})$ played on $C_\mB(X)$. 
\end{theorem}

\begin{proof}

Let $\Phi$ be a winning strategy for player II in $\gone(\mathcal{O}_{\mB},\mathcal{O}_{\mB})$ on $X$. For each $A\in\Omega_{\0}$ and $n\in\omega$, let $\mathcal{U}_n(A)=\{ f^{-1}[I_n]:f\in A\}$, which is an open $\mB$-covering for $X$ by the above Lemma. Before we define a winning strategy for player II in $\gone(\Omega_{\0},\Omega_{\0})$, note that without loss of generality we may assume that for every inning $n\in\omega$, player I does not play $A_n\in\Omega_{\0}$ such that $X\in \mU_n(A_n)$. In fact, if player I chooses any $A_n\in \Omega_{\0}$ such that $X\in\mU_n(A_n)$, then player II picks an $f_n\in A_n$ such that $f_n^{-1}[I_n]=X$; if $X\in\mU_n(A_n)$ for infinitely many $n$, then it is easy to see that $\{f_n:n\in\omega\}\in\Omega_{\0}$. So, we may assume that $X\not\in\mU_n(A_n)$ for all $n$.

 We define a winning strategy $\eta$ for player II in the game $\gone(\Omega_{\0},\Omega_{\0})$ on $C_\mB(X)$ as follows: for every inning $n\in\omega$, let $\eta((A_0,\dots,A_n))=f_n\in A_n$, where $f_n^{-1}[I_n]=\Phi((\mU_0(A_0),\dots,\mU_n(A_n)))$ is an open set that belongs to the open $\mB$-covering $\mU_n(A_n)$. We will show that $\eta$ is a winning strategy.

Let $(A_0,f_0, A_1,f_1,\dots,A_n, f_n,\dots)$ be a whole play in the game $\gone(\Omega_{\0},\Omega_{\0})$, where for each $n$, $A_n$ is the move of player I and $f_n=\eta((A_0,\dots,A_n))$ is the answer of player II at the inning $n$. Since we have $X\not\in\mU_n(A_n)$ for all $n\in\omega$, it follows that $(\mU_n(A_n))_{n\in\omega}$ is a sequence of nontrivial open $\mB$-coverings for $X$, hence it is a valid sequence of moves for the player I in $\gone(\mO_{\mB},\mO_{\mB})$. Now, let $C_n=\Phi((\mU_m(A_m):m\leq n))$ for all $n$ and note that $f_n^{-1}[I_n]=C_n$. Since $\Phi$ is a winning strategy for player II, the play $(\mU_0(A_0),C_0,\dots,\mU_n(A_n),C_n,\dots)$ in the game $\gone(\mO_{\mB},\mO_{\mB})$ is won by player II. This yields that $\{C_n:n\in\omega\}$ is a nontrivial open $\mB$-covering for $X$, thus $\{C_n:n\geq j\}\in\mO_{\mB}$ for any $j\in\omega$, since we are excluding only finitely many $C_n$'s. Then, we have $\{f_n:n\in\omega\}\in\Omega_{\0}$, as desired.

Conversely, let $\psi$ be a winning strategy for player II in $\gone(\Omega_{\0},\Omega_{\0})$ on $C_\mB(X)$. For each nontrivial $\mU\in\mO_{\mB}$, let $A(\mU)=\{f\in C_{\mB}(X):\exists U\in\mU(f\upharpoonright X\setminus U\equiv 1)\}$, which is an element of $\mO_{\mB}$ by the above Lemma. Clearly we may suppose that the player I does not play trivial open $\mB$-coverings.

We define a winning strategy $\rho$ for player II in $\gone(\mathcal{O}_{\mB},\mathcal{O}_{\mB})$ on $X$ as follows: if $(\mU_0,\dots,\mU_n)$ is a sequence of nontrivial open $\mB$-coverings played by player I, let $\rho((\mU_0,\dots,\mU_n))=U_n\in\mU_n$, where $U_n$ is such that $\psi((A(\mU_0),\dots,A(\mU_n)))\upharpoonright X\setminus U_n\equiv 1$. We now show that $\rho$ is a winning strategy for player II in $\gone(\mO_{\mB},\mO_{\mB})$.

Let $(\mU_0,U_0,\dots,\mU_n,U_n,\dots)$ be a play in the game $\gone(\mO_{\mB},\mO_{\mB})$, where $\mU_n$ and $U_n=\rho((\mU_0,\dots,\mU_n))$ are the moves of player I and player II at the inning $n$, respectively. Calling $f_n=\psi((A(\mU_0),\dots,A(\mU_n)))$ for each $n$, we have that $$(A(\mU_0),f_0,\dots,A(\mU_n),f_n,\dots)$$ is a valid play in $\gone(\Omega_{\0},\Omega_{\0})$, which is won by player II since $\psi$ is a winning strategy. So, $\{f_n:n\in\omega\}\in\Omega_{\0}$, from which it follows that $\{U_n:n\in\omega\}\in\mO_{\mB}$.
%In fact, for a sequence $(\mU_n)_{n\in\omega}$ of nontrivial open $\mB$-coverings for $X$ played by player I in $\gone(\mO_{\mB},\mO_{\mB})$, we have that $(A(\mU_n))_{n\in\omega}$ is a sequence of elements in $\Omega_{\0}$ played by player I in $\gone(\Omega_{\0},\Omega_{\0})$. Since $\psi$ is a winning strategy for player II, it follows that $\{\psi((A(\mU_m):m\leq n)):n\in\omega\}\in\Omega_{\0}$, which give us $\{\rho((\mU_m:m\leq n)):n\in\omega\}\in\mO_{\mB}$.
\end{proof}

Actually, with similar adaptations of the arguments and definitions given by Scheepers in \cite{Scheepers1}, one can prove the following:

\begin{theorem}
Let $X$ be a Tychonoff space, $\mB$ be a bornology with compact base on $X$ and $\alpha\geq\omega$ be a countable ordinal. Player II has a winning strategy in $\gone^{\alpha}(\mO_{\mB},\mO_{\mB})$ played on $X$ if, and only if, player II has a winning strategy in $\gone^{\alpha}(\Omega_{\0},\Omega_{\0})$ played on $C_{\mB}(X)$.
\end{theorem}

By puting $\mB=\mF$ on the above theorem we obtain the original Scheepers' result. If $\mB=\frak{K}$ and $\alpha=\omega$, then we have the following:

\begin{corollary}\label{res1}
Let $X$ be a Tychonoff space. Player II  has a winning strategy
in $\gone(\mK,\mK)$ played on $X$  if, and only if,  player
II has a winning strategy in
$\gone(\Omega_{\0},\Omega_{\0})$ played on $C_k(X)$. 
\end{corollary}

The next theorem lead us to the generalization of Theorem \ref{Uspenskii} that we are looking for. First, observe that since a bornology $\mB$ is a structure on the set $X$, it makes sense to use $l_{\mB}(X_\kappa)$ to denote the $\mB$-Lindel\"of degree of the $\kappa$-modification of $X$, i.e., the least transfinite cardinal $\lambda$ such that for every $\mB$-covering $\mU$ for $X$ made by $G_{\kappa}$-sets there exists a $\mB$-subcovering $\mU'\subset \mU$ with $|\mU'|\leq\lambda$. We would like to thank Angelo Bella for his suggestions that improved one of the implications of our original statement.

\begin{theorem}\label{Uspenskii1}
For a Tychonoff space $X$ and a bornology $\mB$ with compact base on $X$, $\kappa\in Sp(C_{\mB}(X))$ if, and only if, $l_{\mB}(X_\kappa)\leq \kappa$.
%\begin{enumerate}
%\item $\kappa\in Sp(C_\mB(X))$;
%\item $l_{\mB}(X_\kappa)\leq\kappa$, then $\kappa\in Sp(C_{\mB}(X))$.
%\end{enumerate}
\end{theorem}

\begin{proof}
Let $\mathcal{G}$ be a nontrivial $\mB$-covering for $X$ made by $G_{\kappa}$-sets which does not contain any $\mB$-subcovering of cardinality less than or equal to $\kappa$. We will show that $C_{\mB}(X)$ is not a productively $\kappa$-tight space. Since $C_{\mB}(X)$ is a homogenous space, by Theorem \ref{ArhanSp} it follows that it is equivalent to show that $C_{\mB}(X)$ is $\kappa$-singular at $\0$.

For each $B\in\mB$, there exists $G\in\mathcal{G}$ such that $\overline{B}\subset G$ and $G=\bigcap\mU(G)$, where $\mU(G)$ is a collection of open subsets of $X$ and $|\mU(G)|\leq\kappa$. Without loss of generality, we may assume that $\mU(G)$ is closed under finite intersections. Since $X$ is a Tychonoff space, for any $U\in\mU(G)$ we may fix a function $f_{B,U}\in C(X)$ satisfying $f_{B,U}\upharpoonright \overline{B}\equiv 0$ and $f_{B,U}\upharpoonright (X\setminus U)\equiv 1$. Now, define $A_{B,G,U}=\{f_{B,V}:V\in\mU(G),V\subset U\}$ and let $\mA_{B,G}=\{A_{B,G,U}:U\in\mU(G)\}$. We claim that $\mathbb{P}=\{\mA_{B,G}:B\in\mB,G\in\mathcal{G}, \overline{B}\subset G\}$ is $\kappa$-singular at $\0$.

By construction, each member of $\mA_{B,G}$ has cardinality at most $\kappa$. Furthermore, given $U_0,\dots,U_{n}\in\mU(G)$ for some $G\in\mathcal{G}$ with $\overline{B}\subset G$,  we have that $U=\bigcap_{i<n+1}U_i\in\mU(G)$ and then $f_{B,U}\in A_{B,G,U_0}\cap\dots\cap A_{B,G,U_{n}}$, from which it follows that $\mathcal{A}_{B,G}$ is a centered family. Moreover, for an arbitrary neighborhood $\langle B,\varepsilon\rangle[\0]$ of $\0$, there exists $G\in\mathcal{G}$ such that $\overline{B}\subset G$, thus $A_{B,G,U}\subset\langle B,\varepsilon\rangle[\0]$ for every $U\in\mU(G)$.

Finally, let $\{\mA_{B_\alpha,G_\alpha}:\alpha<\kappa\}\subset\mathbb{P}$. By our assumption about $\mathcal{G}$, it follows that $\{G_\alpha:\alpha<\kappa\}$ is not a $\mB$-covering for $X$ and so there is some $B\in\mB$ satisfying $B\not\subset G_\alpha$ for all $\alpha<\kappa$. Since each $G_\alpha=\bigcap\mU(G_\alpha)$, there exists $U_{\alpha}\in \mU(G_\alpha)$ such that $B\setminus U_\alpha\ne\emptyset$. Note that by the way we defined the functions in $A_{B_{\alpha},G_{\alpha},U_{\alpha}}$, it follows that each element of $A_{B_{\alpha},G_{\alpha},U_{\alpha}}$ takes value 1 in some point of $B$. Thus, the neighborhood $\langle B,\frac{1}{2}\rangle[\0]$ does not intersect $A_{B_{\alpha},G_{\alpha},U_{\alpha}}$ for all $\alpha<\kappa$, and so $\0\not\in\overline{\bigcup\{A_{B_{\alpha},G_{\alpha},U_{\alpha}}:\alpha<\kappa\}}$. Therefore, the family $\mathbb{P}$ is $\kappa$-singular at $\0$, as desired.

%Indeed, $\langle B,\frac{1}{2}\rangle[\0]\cap\bigcup_{\alpha<\kappa}A_\alpha=\emptyset$. To see this, note that if $f\in \bigcup_{\alpha<\kappa}A_\alpha$, then $f\in A_{B_\alpha,U_\alpha,n_\alpha}$ for some $\alpha$, hence $f=f_{B_\alpha,U_\alpha,m}$ for some  $m\geq n_\alpha$ and, by construction, $f\upharpoonright X\setminus U_\alpha(m)\equiv m+1$. Since $U_\alpha(m)\subset U_\alpha(n_\alpha)$ and $B\not\subset U_\alpha(\alpha_n)$, it follows that $|f(x)|>\frac{1}{2}$ for some $x\in B$, i.e., $f\not\in\langle B,\frac{1}{2}\rangle[\0]$. This proves $(1)$.

Conversely, suppose that every $\mB$-covering for $X$ made by $G_\kappa$-sets has a $\mB$-subcovering of cardinality less than or equal to $\kappa$. We show that $C_{\mB}(X)$ is not $\kappa$-singular at $\0$. Let $\mP=\{\mathcal{A}_{\alpha}\}_{\alpha\in I}$ be a collection satisfying conditions $(a)',(b)$ and $(c)$ in the definition of $\kappa$-singular family. We will show that condition $(d)$ does not hold.

For each $\alpha \in I$, $A\in \mathcal{A}_{\alpha}$ and $n\in\omega$, consider the set $\mathcal{U}_{A,n}=\left\{f^{-1}\left[I_n\right]:f\in A\right\},$ and let $\frak{U}_n=\{\bigcap\mathcal{U}_{A,n}:\alpha \in I$ and $A\in\mathcal{A}_{\alpha}\}$. We claim that $\frak{U}_n$ is a $\mB$-covering for $X$ made by $G_\kappa$-sets. Indeed, given $B\in\mB$, we have $\langle B,\frac{1}{n+1}\rangle[\0]$ a neighborhood of $\0$ and, by condition $(c)$, it follows that there are $\alpha \in I$ and $A\in\mathcal{A}_{\alpha}$ ($|A|\leq \kappa$ by condition $(a)'$) such that $A\subset \langle B,\frac{1}{n+1}\rangle[\0]$, so $\mU_{A,n}$ has cardinality at most $\kappa$ and $B\subset \bigcap\mathcal{U}_{A,n}$.

By the hypothesis, each $\frak{U}_n$ has a $\mB$-subcovering $\frak{U}'_n=\{U_{\lambda,n}:\lambda<\kappa\}$. Note that for each $\lambda<\kappa$ and $n\in\omega$ we may choose $\alpha_{\lambda,n}\in I$ and $C_{\lambda,n}\in \mathcal{A}_{\alpha_{\lambda,n}}$ with $U_{\lambda,n}=\bigcap\mathcal{U}_{C_{\lambda,n},n}$. We claim that $\{\mathcal{A}_{\alpha_{\lambda,n}}:\lambda<\kappa,n\in\omega\}$ witnesses that $(d)$ does not hold.

In fact, for each $\lambda<\kappa$ and $n\in\omega$ choose $A_{\lambda,n}\in\mA_{\alpha_{\lambda,n}}$, and let $\langle B,\varepsilon\rangle[\0]$ be a neighborhood of $\0$. Fix $n\in\omega$ with $\frac{1}{n+1}<\varepsilon$. Since $\frak{U}'_{n}$ is a $\mB$-subcovering, there exists an $U_{\lambda,n}$ with $B\subset U_{\lambda,n}$. So, we have $C_{\lambda,n}\subset\langle B,\varepsilon\rangle[\0]$, from which it follows that $A_{\lambda,n}\cap\langle B,\varepsilon\rangle[\0]\ne\emptyset$, since $A_{\lambda,n}\cap C_{\lambda,n}\ne \emptyset$ by condition $(b)$. This ends the proof.
\end{proof}

By making $\kappa=\aleph_0$ on the above theorem, we obtain the following:

\begin{corollary}
Let $X$ be a Tychonoff space and let $\mB$ be a bornology with compact base on $X$. $C_{\mB}(X)$ is productively countably tight if, and only if, every $\mB$-covering for $X$ made by $G_{\delta}$-sets has a countable $\mB$-subcovering.
\end{corollary}

\section{Further applications}

Recall that a space $X$ is said to be an {\bf Alster space} if for every $K$-covering for $X$ made by $G_\delta$'s there is a countable subcovering. In \cite{Alster}, it was showed that if $X$ is an Alster space, then $X\times Y$ is a Lindel\"of space for any Lindel\"of space $Y$, i.e., $X$ is a {\bf productively Lindel\"of space}, and under {\bf CH}, every productively Lindel\"of space with a base of cardinality at most $\omega_1$ is an Alster space.

So, it is natural to say that a topological space $X$ is {\bf strongly Alster} if every $K$-covering for $X$ made by $G_{\delta}$-sets has a countable $K$-subcovering. Since strongly Alster condition trivially implies Alster condition, which in turns implies productively Lindel\"ofness, we have obtained the following

\begin{corollary}\label{res2}
For a Tychonoff space $X$, $C_{k}(X)$ is productively countably tight if, and only if, $X$ is a strongly Alster space.
\end{corollary}

\begin{corollary}\label{ckpl}Let $X$ be a Tychonoff space. If $C_k(X)$ is productively countably tight, then $X$ is productively Lindel\"of.
\end{corollary}

Indeed, the strongly Alster condition is stronger than the Alster one. To see this, note that in a space in which every compact set is a $G_\delta$-set, the Alster condition is equivalent to $\sigma$-compactness, while strongly Alster condition is equivalent to hemicompactness. Thus, $C_k(\QQ)$ is not productively countably tight, since $\QQ$ is not hemicompact, while $C_p(\QQ)$ is productively countably tight, since $C_p(\QQ)$ is first countable. On the other hand, since the space $P$ of the irrational numbers is not an Alster space, it follows that $C_{\mB}(P)$ is not productively countably tight for every bornology $\mB$ on $P$ with compact base (cf. Corollary \ref{lastcor}). 

We denote by $\sone(\mA,\mC)$ the following selection principle: for each sequence $(A_n)_{n\in\omega}$ of elements of $\mA$ there exists a sequence $(a_n)_{n\in\omega}$ with $a_n\in A_n$ for all $n\in\omega$ such that $\{a_n:n\in\omega\}\in\mC$. In \cite{Kocinac}, Ko\v cinac proved\footnote{A slightly adaptation of Ko\v cinac's proof shows that for any bornology $\mB$ with compact base on $X$, $C_{\mB}(X)$ satisfies $\sone(\Omega_\0,\Omega_\0)$ iff $X$ satisfies $\sone(\mO_{\mB},\mO_\mB)$.} the next result.
\begin{theorem}[Ko\v cinac, \cite{Kocinac}, Theorem 2.2]\label{Kocinac} For a Tychonoff space $X$, $C_k(X)$ satisfies $\sone(\Omega_\0,\Omega_\0)$ if, and only if, $X$ satisfies $\sone(\mK,\mK)$.
\end{theorem}

Also, in \cite{Aurichi}, Aurichi and Bella had obtained the following:

\begin{theorem}[Aurichi and Bella, \cite{Aurichi}, Corollary 2.4 and Theorem 2.5]\label{Aurichi} Let $X$ be a Tychonoff space. If player II has a winning strategy in $\gone(\Omega_{x},\Omega_{x})$ played on $X$, then $\aleph_0\in Sp(x,X)$. Also, if $\aleph_0\in Sp(x,X)$, then $X$ satisfies $\sone(\Omega_{x},\Omega_{x})$.
\end{theorem}

By changing $X$ for $C_k(X)$ on the above theorem, we obtain some interesting implications by using the results proved so far, that we summarize on the next diagram. Note that each arrow has the number of the results from which the implication follows, also, we write II $\uparrow \gone(\mA,\mC)(Y)$ to mean that player II has a winning strategy in $\gone(\mA,\mC)$ played on the space $Y$.

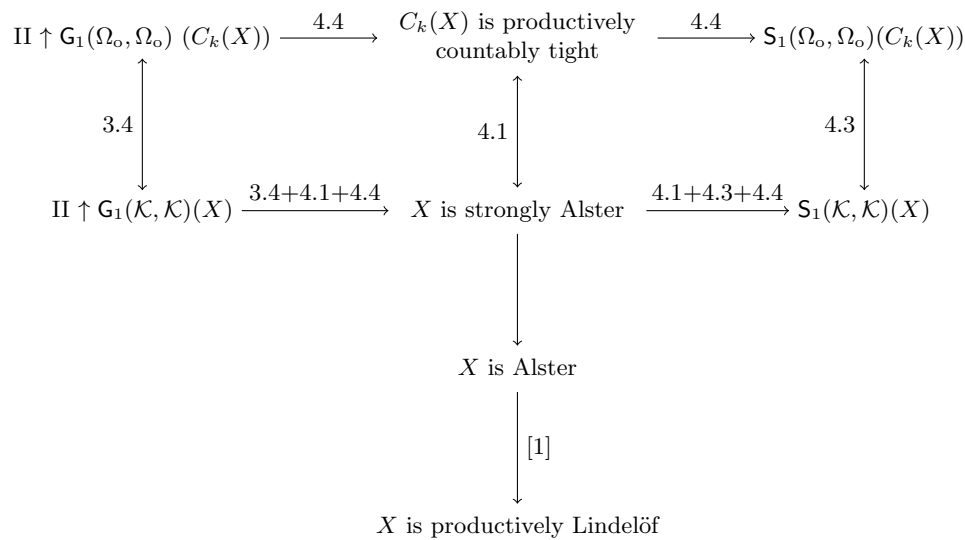
\small{
\begin{center}

\begin{tikzpicture}\label{diagram}

\matrix(m)[matrix of nodes, row sep=15mm, column sep=10mm,
  jump/.style={text width=20mm,anchor=center},
  txt/.style={anchor=center},
]
{
II $\uparrow\gone(\Omega_{\0}, \Omega_{\0})$ ($C_k(X)$) & 
\begin{tabular}{c}
$C_k(X)$ is productively\\
 countably tight  
\end{tabular}
 &
$\sone(\Omega_\0,\Omega_\0)(C_k(X))$\\
II $\uparrow\gone(\mK,\mK)(X)$ &
\begin{tabular}{c}
$X$ is strongly Alster
\end{tabular} &
$\sone(\mK,\mK)(X)$\\
&\begin{tabular}{c}
$X$ is Alster
\end{tabular}&\\
 &\begin{tabular}{c}
$X$ is productively Lindel\"of
\end{tabular} &\\
};

\draw[<->] (m-2-2) to node[auto]{\ref{res2}} (m-1-2);
\draw[<->] (m-2-1) to node[auto]{\ref{res1}} (m-1-1);
\draw[<->] (m-2-3) to node[auto]{\ref{Kocinac}} (m-1-3);
\draw[->] (m-1-1) to node[auto]{\ref{Aurichi}} (m-1-2);
\draw[->] (m-1-2) to node[auto]{\ref{Aurichi}} (m-1-3);
\draw[->] (m-2-2) to (m-3-2);
\draw[->] (m-3-2) to node[auto]{\cite{Alster}} (m-4-2);
\draw[->] (m-2-1) to node[auto]{\ref{res1}+\ref{res2}+\ref{Aurichi}} (m-2-2);
\draw[->] (m-2-2) to node[auto]{\ref{res2}+\ref{Kocinac}+\ref{Aurichi}} (m-2-3);

\end{tikzpicture}
\end{center}}

\normalsize
Since any bornology $\mB$ with compact base is a subset of $\frak{K}$, it follows that any $K$-covering for $X$ is also a $\mB$-covering. Thus, Corollary \ref{ckpl} is a particular case of the following

\begin{corollary}\label{lastcor}
For a Tychonoff space $X$, if there exists a bornology $\mB$ in $X$ with a compact base such that $C_\mB(X)$ is productively countably tight, then $X$ is an Alster space, and hence $X$ is productively Lindel\"of.
\end{corollary}

\end{document}